\documentclass[11pt]{article}
\usepackage{graphicx}
\usepackage{amsmath}
\usepackage{graphicx}
\usepackage{verbatim}
\usepackage{color}
\usepackage{subfigure}
\usepackage{float}
\usepackage{authblk}
\usepackage{amssymb,enumerate,latexsym}
\usepackage{textcomp}
\usepackage{eurosym}
\usepackage{csquotes}
\usepackage{hyperref}
\newtheorem{theorem}{Theorem}

\newtheorem{corollary}{Corollary}

\newtheorem{lemma}{Lemma}

\newenvironment{proof}{\paragraph{Proof:}}{\hfill$\square$}

\begin{document}
\title{\textbf{Fibonacci or Lucas numbers that are products of two Lucas numbers or two Fibonacci numbers}}
\author[,1]{Ahmet Da\c{s}demir\thanks{Corresponding author}}
\author[2]{Ahmet Emin}
\affil[1]{Department of Mathematics, Faculty of Science, Kastamonu University, Kastamonu, Turkey, E-mail: ahmetdasdemir37@gmail.com}
\affil[2]{Department of Mathematics, Faculty of Science, Karabük University, Karabük, Turkey, E-mail: ahmetemin@karabuk.edu.tr}

\date{\vspace{-5ex}}
\maketitle
\begin{abstract} 
This contribution presents all possible solutions to the Diophantine equations $F_k=L_mL_n$ and $L_k=F_mF_n$. To be clear, Fibonacci numbers that are the product of two arbitrary Lucas numbers and Lucas numbers that are the product of two arbitrary Fibonacci numbers are determined herein. The results under consideration are proven by using Dujella-Pethö lemma in coordination with Matveev's theorem. All common terms of the Fibonacci and Lucas numbers are determined. Further, the Lucas-square Fibonacci and Fibonacci-square Lucas numbers are given.


\end{abstract}

\section{Introduction}\label{Int}

Let $\left\lbrace  F_n \right\rbrace_{n \geqslant 0} $ and $\left\lbrace  L_n \right\rbrace_{n \geqslant 0} $ be the $n\text{th}$ terms of the Fibonacci and Lucas numbers, which can be produced by utilizing the recurrence relation $F_{n+1} = F_{n} + F_{n-1}$ and $L_{n+1} = L_{n} + L_{n-1}$ for all integers $n \geqslant 2$ with the initial conditions $\left( F_{0} , F_{1} \right) = \left( 0 , 1 \right) $ and $\left( L_{0} , L_{1}\right) = \left( 2 , 1 \right) $, respectively. It is seen that both the Fibonacci and Lucas numbers are a second-order integer sequence satisfying $ {x^2} - x - 1 = 0 $. By considering this algebraic equation with the mentioned initial conditions, one can develop the following Binet's formulas for $\forall n \in \mathbb{N}$:
\begin{equation}\label{I1}
{F_n} = \frac{{{\alpha ^n} - {\beta ^n}}}{{\alpha  - \beta }}\,\,\text{and}\,\,{L_n} = {{\alpha ^n} + {\beta ^n}},
\end{equation} 
where $ \alpha = \frac{{1 + \sqrt 5 }}{2} $ and $ \beta = \frac{{1 - \sqrt 5 }}{2} $. More detail can be referenced to \cite{Vor}, \cite{Vaj}, and \cite{Kos}.

It can be stated immediately that the above integer sequences are the most beloved subject of mathematics and are paid great attention by almost all branches of modern sciences. Today, the respective results and discussions are expanding to an exciting aspect: finding all possible solutions to Diophantine equations, including special integer sequences, i.e., Fibonacci, Pell, or Jacobsthal numbers, etc. In \cite{Mar}, Marques investigated the Fibonacci numbers that can be expressible in terms of the generalized Cullen and Woodall numbers. In \cite{Cha1}, Chaves and Marques determined all terms of $k$-generalized Fibonacci numbers, which are the sum of the powers of the consecutive $k$-generalized Fibonacci sequence. In \cite{Bra}, Bravo and Gómez considered $k$-generalized Fibonacci numbers that are the Mersenne numbers. In \cite{Pong}, Pongsriiam found all the Fibonacci and Lucas numbers, which are one away from the product of an arbitrary number of the Fibonacci or Lucas numbers. In \cite{Dda}, Ddamulira et al. solved the Pillai-type problem with $k$–generalized Fibonacci numbers and powers of 2 for $k > 3$. In \cite{Kaf}, Kafle et al. dealt with finding all solutions to the Pell equations related to the product of two Fibonacci numbers. In \cite{Cha2}, Qu and Zeng investigated all Lucas numbers that are concatenations of two repdigits. In \cite{Sia}, Şiar et al. found all Fibonacci or Lucas numbers that are products of two repdigits in base $b$. In \cite{Ala}, Alan and Alan discovered the Mersenne numbers that can be written in terms of the products of two Pell numbers. In \cite{Rih}, Rihane and Togbé obtained terms of $k$-Fibonacci numbers in the arrays of the Padovan or Perrin numbers.

In the open literature, there are a few more specific studies that study the Diophantine-type equations concerning the Fibonacci numbers or other integer sequences. However, both the above brief literature survey and other source work show that integer sequences in the right-hand side and left-hand side of problems under consideration are of different characteristic algebraic equations. For example, Fibonacci or Lucas number vs. Pell number by Alekseyev \cite{Ale}, Fibonacci number vs. Pell number by Ddamulira et al. \cite{Dda2}, generalized Fibonacci number vs. generalized Pell number by Bravo et al. \cite{Bra2}, Fibonacci number vs. Jacobsthal number by Erduvan and Keskin \cite{Erd}, and Leonardo number vs. Jacobsthal number by Bensella and Behlou \cite{Ben}. Motivated by the results of the current literature, in this paper, we address finding problem of all possible solutions to the following Diophantine equations for positive integers $k$, $m$, and $n$ according to the famous Matveev's theorem and Dujella-Pethö lemma:
\begin{equation}\label{I2}
{F_k} = {L_m} {L_n}
\end{equation}
and 
\begin{equation}\label{I3}
{L_k} = {F_m} {F_n}. 
\end{equation}
Here, due to multiplicative symmetry, it is sufficient that the case where $k \geqslant 1$ and $1 \leqslant m \leqslant n$ is considered. However, our equations consist of the Fibonacci and Lucas numbers that have the same characteristic equation. This makes the application of Dujella-Petho lemma impossible because some parameters disappear unlike the solution processes in the current literature. One of the novelties of the paper is to display a new approach to this issue.

It should be noted that in \cite{Car}, Carlitz considered the same problems the first time by employing divisibility properties and some elementary identities. However, the author's results are either incorrect or incomplete. More precisely, the author asserted that while the unique solution of Equation \eqref{I2} is $\left( k,m,n \right) = \left( 8,4,2\right)$ for $m \geqslant n > 1$, Equation \eqref{I3} has no solution for $m \geqslant n > 2$. Further, in \cite{Wan}, based on the elementary properties and inequalities, Wang et al. stated that while Equation \eqref{I3} has no solution, the triple $\left( k,m,n \right) = \left( 4,2,1 \right)$ is one solution to Equation \eqref{I2}. As can be seen, the results of both studies are also contradictory to each other. The results of our paper will both eliminate this deficiency and will end this debate.

\section{Basic Tools}\label{Bas}

This section introduces essential tools and definitions, lemmas, and theorems required in the rest of the paper. Our proof process is based on Matveev's theorem, which uses the linear forms in logarithms to limit the variables of the problem, and Dujella-Pethö lemma, which allows us to reduce the bounds.

Let $ \eta $ be an algebraic number of degree $ d $ with the minimal polynomial
\begin{equation*}
f\left( X \right): = \sum\limits_{j = 0}^d {{a_j}{x^{d - j}}} = {a_0}{X^d} + {a_1}{X^{d - 1}} +  \cdots  + {a_d} = \prod\limits_{i = 1}^d {\left( {X - {\eta ^{\left( i \right)}}} \right)}  \in Z\left[ X \right],
\end{equation*}
where ${a_0} > 0$,  ${a_i}$'s are relatively prime integers, and $\eta^{\left( i \right)} $ is the $i\text{th}$ conjugate of $\eta$. The logarithmic height, denoted by $h\left( \eta  \right)$, of $ \eta $ is defined by 
\begin{equation*}
h\left( \eta  \right) = \frac{1}{d}\left( {\log \left| {{a_0}} \right| + \sum\limits_{i = 1}^d {\log \left( {\max \left\{ {\left| {{\eta ^{\left( i \right)}}} \right|,1} \right\}} \right)} } \right),
\end{equation*}

Let ${\eta_1},{\eta_2}, \ldots ,{\eta_s}$ be positive algebraic numbers in the real number field $\mathcal{F}$ of degree $D$ and let ${b _1},{b _2}, \ldots ,{b _s}$ be nonzero rational numbers. Introduce the notations
\begin{equation*}
\Lambda : = {\eta _1}^{{b_1}}{\eta _2}^{{b_2}} \ldots {\eta _s}^{{b_s}} - 1{\kern 1pt} {\kern 1pt} {\text{and}}{\kern 1pt} {\kern 1pt} B : = \max \left\{ {\left| {{b_1}} \right|,\left| {{b_2}} \right|, \ldots ,\left| {{b_s}} \right|} \right\}.
\end{equation*}

Let ${A_1},{A_2},\ldots,{A_s}$ be the positive real numbers such as
\begin{equation*}
{A_j} \geqslant \max \left\{ {Dh\left( {{\eta _j}} \right),\left| {\log {\eta _j}} \right|,0.16} \right\}\,\,\text{for all}\,\,j=1,2,\ldots,s. 
\end{equation*}

In this case, we can give the famous Matveev's theorem \cite{Mat} and Dujella-Pethö lemma \cite{Duj}.
\begin{theorem}[Matveev \cite{Mat}]\label{PT1}
The following inequality holds for non-zero $\Lambda$ over real field $\mathcal{F}$: 
\begin{equation*}
\log \left( {\left| \Lambda  \right|} \right) >  - 1.4 \times {30^{s + 3}} \times {s^{4.5}} \times D^2 \times \left( {1 + \log D} \right) \times\left( {1 + \log B} \right) \times{A_1}\times{A_2}\times...\times{A_s}.
\end{equation*}
\end{theorem}

\begin{lemma}[Dujella and Pethö \cite{Duj}]\label{PT3}
Let $M$ be a positive integer, $p/q$ be a convergent of the continued fraction of the irrational $\tau$ such that $q>6M$, and let $A$, $B$, $\tau$ be positive rational numbers with $A>0$ and $B>1$. Let $\varepsilon  = \left\| {\mu q} \right\| - M\left\| {\tau q} \right\|$, where $\left\|  \cdot  \right\|$ is the distance from the nearest integer. If $\varepsilon > 0$, then there is no integer solution $\left(m,n,k\right) $ of inequality 
\begin{equation*}
0 < m\tau  - n + \mu  < A{B^{ - k}}
\end{equation*}
with
\begin{equation*}
m \leqslant M\,\,\text{and}\,\,k \geqslant \frac{{\log \left( {Aq/\varepsilon } \right)}}{{\log B}}.
\end{equation*}
\end{lemma}

The following lemmas will be used later.

\begin{lemma}\label{PT4}
Let $n$ be non-negative integer. Then, 
\begin{equation}\label{P1}
{\alpha^{n - 2}} \leqslant {F_n} \leqslant {\alpha^{n - 1}},
\end{equation}
\begin{equation}\label{P2}
{\alpha^{n - 1}} \leqslant {L_n} \leqslant 2{\alpha^{n}},
\end{equation}
\begin{equation}\label{P3}
{\left| \beta  \right|^{ - \left( n - 2 \right) }} \leqslant {F_n} \leqslant {\left| \beta  \right|^{ - \left( n - 1\right) }},
\end{equation}
and
\begin{equation}\label{P4}
{\left| \beta  \right|^{ - \left( n - 1 \right) }} \leqslant {L_n} \leqslant {\left| \beta  \right|^{ - \left( n + 1\right) }}.
\end{equation}
\end{lemma}
\begin{proof}
The proof can be made by using the induction method on $n$.
\end{proof}

\begin{lemma}[Ddamulira et al. \cite{Dda2}]\label{PT5}
For all $x \in \left( { - \frac{1}{2},\frac{1}{2}} \right)$, $\left| x \right| < 2\left| {{e^x} - 1} \right|$ is satisfied.
\end{lemma}

\section{Main Results}\label{Mai}

In this section, we will present all solutions to Equations \eqref{I2} and \eqref{I3} and will prove our results.

\begin{theorem}\label{T1}
Let $ k $, $ m $, and $ n $ be a non-zero integer. Then, 
\begin{itemize}
\item Equation \eqref{I2} is satisfied only for the triples of 
\begin{equation}\label{M1}
\left( {k,m,n} \right) \in \left\lbrace  \left( {1,1,1} \right), \left( {2,1,1} \right), \left( {4,1,2} \right),\left( {8,2,4} \right) \right\rbrace. 
\end{equation}

\item Equation \eqref{I3} holds only for the triples of  
\begin{equation}\label{M2}
\left( {k,m,n} \right) \in \left\lbrace \left( {1,1,1} \right),\left( {1,1,2} \right),\left( {1,2,2} \right),\left( {2,1,4} \right),\left( {2,2,4} \right),\left( {3,3,3} \right)\right\rbrace . 
\end{equation}
\end{itemize}
\end{theorem}
\begin{proof}

Here, to reduce the size of the paper, we will only share a detailed proof for Equation \eqref{M1}, neglecting that of Equation \eqref{M2}.

From Equation \eqref{I2} and Lemma \ref{PT4}, we can write 
\begin{equation*}
{\alpha ^{k - 2}} \leqslant {F_k} = {L_n}{L_m} \leqslant {\left| \beta  \right|^{ - n - m - 2}}
\end{equation*}
and naturally
\begin{equation*}
\left( {k - 2} \right)\log \alpha  \leqslant  - \left( {n + m + 2} \right)\log \left| \beta  \right| \Rightarrow k \leqslant 2 - \left( {n + m + 2} \right)\frac{{\log \left| \beta  \right|}}{{\log \alpha }} \Rightarrow k < 4n.
\end{equation*}
Considering the Binet's formulas in Equation \eqref{I1}, we can arrange Equation \eqref{I2} as follows:
\begin{equation}\label{M3}
{\Lambda _1}: = \left| {{\alpha ^{ - k}}{{\left| \beta  \right|}^{n + m}}\sqrt 5  - 1} \right| < \frac{8}{{{\alpha ^{2m}}}}.
\end{equation}
In this case, we can consider the case $s = 3$, ${\eta _1} = \alpha$, ${\eta _2} = \left| \beta  \right|$, ${\eta _3} = \sqrt{5}$, ${b_1} = -k$,  ${b_2} = n + m$, and ${b_3} = 1$ in Theorem \ref{PT1}. To be clear, ${\eta _1},{\eta _2},{\eta _3} \in \mathbb{Q}\left( {\sqrt 5 } \right)$ and $\mathcal{F} = \mathbb{Q}\left( {\sqrt 5 } \right)$ of degree ${D} = 2$. Here, since $ {\alpha ^k} \cdot {\left| \beta  \right| ^{ - n - m}} = \sqrt 5 $ is not satisfied when computing the square of its both sides, ${\Lambda _1} \ne 0$. In addition,   
\begin{equation*}
h\left( {{\eta _1}} \right) =\frac{1}{2}\log \alpha ,\,\,h\left( {{\eta _2}} \right) = \frac{1}{2}\log \alpha,\,\, h\left( {{\eta _3}} \right) = \log \sqrt{5},\,\,{A_1} = \log \alpha,\,\, {A_2} = \log \alpha, \,\, \text{and}\,\, {A_3} = 2\log \sqrt{5}.
\end{equation*}
Further, for $B = 4n$, $B \geqslant \max \left\{ {\left| { - k} \right|,n + m,1} \right\}$. Then, with these values, Theorem \ref{PT1} implies that
\begin{equation}\label{M4}
\log ({\Lambda _1}) >  - 3.62 \times {10^{11}}\times \left( {1 + \log{4n}} \right).
\end{equation}
Also, with Equation \eqref{M3}, we obtain
\begin{equation}\label{M5}
\log ({\Lambda _1}) < \log 8 - 2m\log \alpha. 
\end{equation}
As a result, we get
\begin{equation}\label{M6}
m < 3.77 \times {10^{11}}\times\left( {1 + \log 4n} \right).
\end{equation}

Further, coming back to Equation \eqref{I2}, after some mathematical arrangements, we can write
\begin{equation}\label{M7}
{\Lambda _2} : = \left| {{\alpha ^{ - k}}{\left| \beta  \right| ^n}\left( {\sqrt 5 {L_m}} \right) - 1} \right| < {\frac{{33}}{{{\alpha ^n}}}},
\end{equation}
which implies that $s = 3$, ${\eta _1} = \alpha$, ${\eta _2} = \left| \beta  \right|$, ${\eta _3} = \sqrt 5 {L_m}$, ${b_1} = -k$, ${b_2} = n$, and ${b_3} = 1$. Here, ${\eta _1},{\eta _2},{\eta _3} \in \mathbb{Q}\left( {\sqrt 5 } \right)$ and $\mathcal{F} = \mathbb{Q}\left( {\sqrt 5 } \right)$ of degree ${D} = 2$. On the other hand, one can prove that ${\Lambda _2} \ne 0$ by applying the same procedure above. As a result, 
\begin{equation*}
h\left( {{\eta _1}} \right) =\frac{1}{2}\log \alpha ,\,\,h\left( {{\eta _2}} \right) = \frac{1}{2}\log \alpha,\,\,{A_1} = \log \alpha,\,\, \text{and}\,\, {A_2} = \log \alpha.
\end{equation*}
Further, since $\eta _3$ is the root of the polynomial ${x^2} - 5L_m^2$, $h\left( {{\eta _3}} \right) = \log \left( {\sqrt 5 {L_m}} \right)$ and ${A_3} = 6m\log \alpha $. In addition, $B \geqslant \max \left\{ {\left| { - k} \right|,n,1} \right\}$ for $B = 4n$. From Theorem \ref{PT1}, we can write
\begin{equation}\label{M8}
\log ({\Lambda _2}) >  - 6.49 \times {10^{11}}\times m \times\left( {1 + \log 4n} \right). 
\end{equation}
Solving Equations \eqref{M6} and \eqref{M8} together, we get
\begin{equation}\label{M9}
\log ({\Lambda _2}) >  - 2.45 \times {10^{23}}{\left( {1 + \log 4n} \right)^2}.	
\end{equation}
Also, from Equation \eqref{M7}, we obtain
\begin{equation}\label{M10}
\log ({\Lambda _2}) < \log 33 - n\log \alpha.
\end{equation}
Considering Equations \eqref{M9} and \eqref{M10}, we find  
\begin{equation}\label{M11}
n < 2.18 \times {10^{27}}.
\end{equation}

After applying a similar process into Equation \eqref{I3}, we determine the bounds
\begin{equation}\label{M12}
m < 7.52 \times {10^{11}}\times\left( {1 + \log 4n} \right)\,\,\text{and}\,\,n < 2.25 \times {10^{27}}.
\end{equation}
Summing up, it is sufficient that we consider the following lemma in order to complete the proof.

\begin{lemma}\label{T2}
Both Equations \eqref{I2} and \eqref{I3} are satisfied for all the ordered triples of $\left( {k,m,n} \right)$ over the ranges $k < 4n$, $1 \leqslant m \leqslant n$, and $n < 2.25 \times {10^{27}}$.
\end{lemma}

According to Lemma \ref{T2}, there is a finite number of solutions. But, the bounds are very huge, and thereby, we must obtain a more favorable condition. To do this, we will use Dujella-Pethö lemma for four different cases.

\textbf{Case I:} Introducing the notation 
\begin{equation*}
{\Gamma _1}: =  - k\log \alpha  + \left( {n + m} \right)\log \left| \beta  \right| + \log \sqrt 5.
\end{equation*}
we can write
\begin{equation*}
{\Lambda _1}: = \left| {\exp \left( {{\Gamma _1}} \right) - 1} \right| < \frac{8}{{{\alpha ^{2m}}}}.
\end{equation*}
From Lemma \ref{PT5}, we obtain
\begin{equation*}
0 < \left| {k\frac{{\log \alpha }}{{\log \left| \beta  \right|}} - \left( {n + m} \right) + \frac{{\log \left( {1/\sqrt 5 } \right)}}{{\log \left| \beta  \right|}}} \right| < \frac{{34}}{{{\alpha ^{2m}}}}
\end{equation*}
When applying Dujella-Pethö lemma into the last inequality by considering $M = 9.1 \times {10^{27}}$ $\left( M > 4n > k \right)$ and $\tau  = \frac{{\log \alpha }}{{\log \left| \beta  \right|}}$, computing the continued fraction expansions of $\tau$ yields
\begin{equation*}
\frac{{{p_{47}}}}{{{q_{47}}}} = \frac{{13949911361108065346183311454}}{{{92134223612043233793615516979}}}.
\end{equation*}
This means that $6M < {q_{47}} = 92134223612043233793615516979$. As a result, we obtain
\begin{equation*}
\varepsilon : = \left\| {\mu {q_{47}}} \right\| - M\left\| {\tau {q_{47}}} \right\|,\,\, \varepsilon  > 0.486
\end{equation*}
In this case, for the case where $A: = 34$, $B: = \alpha^2 $, and $k: = m$ in Lemma \ref{PT3}, we conclude that $m \leqslant 73$.

\textbf{Case II:} Assume that $5 < m \leqslant 107$. Considering
\begin{equation*}
{\Gamma _2}: =  - k\log \alpha  + n\log \left| \beta  \right| - \log \left( {\frac{1}{{\sqrt 5 {L_m}}}} \right),
\end{equation*}
we have
\begin{equation*}
{\Lambda _2}: = \left| {\exp \left( {{\Gamma _2}} \right) - 1} \right| < \frac{{33}}{{{\alpha ^n}}}
\end{equation*}
and from Lemma \ref{PT5}, we can write
\begin{equation*}
0 < \left| {k\frac{{\log \alpha }}{{\log \left| \beta  \right|}} - n + \frac{{\log \left( {1/\left( \sqrt 5 {L_m}\right) } \right)}}{{\log \left| \beta  \right|}}} \right| < \frac{{138}}{{{\alpha ^n}}}.
\end{equation*}
For the case where $M = 9.1 \times {10^{27}}$ $\left( {M > 4n > k} \right)$ and $\tau  = \frac{{\log \alpha }}{{\log \left| \beta  \right|}}$, computing the continued fraction expansions of $\tau$ gives 
\begin{equation*}
\frac{{{p_{47}}}}{{{q_{47}}}} = \frac{{13949911361108065346183311454}}{{{92134223612043233793615516979}}}
\end{equation*}
and, thus, $6M < {q_{47}} = 92134223612043233793615516979$. In this case, 
\begin{equation*}
{\varepsilon _m}: = \left\| {{\mu _m}{q_{47}}} \right\| - M\left\| {\tau {q_{47}}} \right\|,\,\,{\varepsilon _m} > 0.034, \,\,{\mu _m} = \frac{{\log \left( {1/\sqrt 5 {L_m}} \right)}}{{\log \left| \beta  \right|}}.
\end{equation*}
As a result, taking $A: = 138$, $B: = \alpha $, and $k: = n$ in Lemma \ref{PT3}, we obtain that $n \leqslant 160$.

It should be noted that applying a similar investigation into Equation \eqref{I3}, we obtain the bounds in which $m \leqslant 75$ and $n \leqslant 153$. Then, we can compose a unique looping in Mathematica$^\copyright$ over the range $m \leqslant 75$ and $n \leqslant 160$ to determine all possible solutions to both Equations \eqref{I2} and \eqref{I3}. So, running our Pc algorithm validates Theorem \ref{T1}. This exhausts the proof.

\end{proof}

A simple observation of the outcomes of Theorem \ref{T1} reveals the following inferences.

\begin{corollary}\label{T3}
All common terms of the Fibonacci and Lucas numbers are 1 and 3.
\end{corollary}
\begin{proof}
For the case where $m=1$, Equation \eqref{I2} is reduced to $F_k=L_n$. In this case, the result follows from Theorem \ref{T1}.
\end{proof}

\begin{corollary}\label{T4}
The only Lucas-square Fibonacci numbers are $F_1=L_1^2=1$ and $F_2=L_1^2=1$. 
\end{corollary}
\begin{proof}
When $m=n$, Equation \eqref{I2} is reduced to $F_k=L_n^2$. From Theorem \ref{T1}, the result can be drawn.
\end{proof}

\begin{corollary}\label{T5}
The only Fibonacci-square Lucas numbers are $L_1=F_1^2=1$, $L_1=F_2^2=1$, and $L_3=F_3^2=4$. 
\end{corollary}
\begin{proof}
Taking $m=n$ in Equation \eqref{I3} into account, the proof is easily obtained.
\end{proof}

\end{document}